\documentclass[11pt,reqno,a4paper]{article}
\usepackage{lmodern} 
\usepackage[T1]{fontenc} 
\usepackage{float}
\usepackage{comment}
\usepackage{tabularx}
\usepackage{adjustbox}
\usepackage{indentfirst}
 
\usepackage{enumerate}

\usepackage{amsmath,amssymb,amsthm,amsfonts,amscd,authblk,dsfont,cancel,diagbox,
enumerate,epsfig,eucal,fancyhdr,latexsym,lmodern,mathrsfs,mathtools,multirow,
musicography,skull,simplewick,subcaption,tikz-cd}
\usetikzlibrary{matrix,arrows,decorations.pathmorphing}
 \usepackage{relsize}
\usepackage[normalem]{ulem} 

\makeatletter
\newcommand{\pushright}[1]{\ifmeasuring@#1\else\omit\hfill$\displaystyle#1$\fi\ignorespaces}
\newcommand{\pushleft}[1]{\ifmeasuring@#1\else\omit$\displaystyle#1$\hfill\fi\ignorespaces}
\makeatother

\usepackage{hyperref}
\hypersetup{
pdftitle={},%
pdfauthor={},%
pdfsubject={},%
pdfkeywords={},%
colorlinks=true,%
linkcolor=black,%
urlcolor=black,
linktoc={},
linktocpage={false},%
pageanchor={},
citecolor={black},
}

\usepackage[english]{babel}
\usepackage[utf8]{inputenc}
\usepackage[margin=2.35cm]{geometry}

\usepackage{url}

\usepackage{cite}

\usepackage{titlesec}
\titleformat{\section}{\large\bfseries\filcenter}{\thesection}{1em}{}
\titleformat{\subsection}{\bfseries}{\thesubsection}{1em}{}
\let\oldbibliography\thebibliography
\renewcommand{\thebibliography}[1]{\oldbibliography{#1}
\setlength{\itemsep}{0\baselineskip}}
\makeatletter
\renewcommand\@makefntext[1]{\leftskip=2em\hskip-0.5em\@makefnmark#1}
\makeatother
\allowdisplaybreaks[1]

\catcode`@=12

\renewcommand\thanks[1]{%
  \begingroup
  \renewcommand\thefootnote{}\footnote{#1}%
  \addtocounter{footnote}{-1}%
  \endgroup
}
\usepackage{tikz-cd}

\tikzset{%
  symbol/.style={
    draw=none,
    every to/.append style={
      edge node={node [sloped, allow upside down, auto=false]{$#1$}}
    },
  },
}

\usepackage{amsmath,amssymb,amsthm}
\usepackage{tensor} 

\newtheorem*{main-theorem}{Main Theorem}
\newtheorem{theorem}{Theorem}[section]

\newtheorem{proposition}[theorem]{Proposition}
\theoremstyle{definition}
\newtheorem{remark}[theorem]{Remark}

\newtheorem{definition}[theorem]{Definition}
\newtheorem{example}[theorem]{Example}
\newtheorem{examples}[theorem]{Examples}
\renewcommand{\:}{\colon} 
\renewcommand{\i}{\mathrm{i}} 

\renewcommand{\epsilon}{\varepsilon}
\newcommand{\NN}{\mathbb{N}}

\newcommand{\RR}{\mathbb{R}}
\newcommand{\CC}{\mathbb{C}}

\newcommand{\A}{\mathsf{A}}

\newcommand{\C}{\mathsf{C}}
\newcommand{\D}{\mathsf{D}}

\renewcommand{\L}{\mathsf{L}}

\renewcommand{\H}{\mathsf{H}}
\renewcommand{\P}{\mathsf{P}}
\newcommand{\K}{\mathsf{K}}

\newcommand{\V}{\mathsf{V}}

\renewcommand{\S}{\mathsf{S}}

\renewcommand{\d}{\mathrm{d}}
\newcommand{\I}{\mathrm{I}}

\renewcommand{\c}{\mathrm{c}}
\newcommand{\M}{\mathsf{M}}	
\newcommand{\T}{\mathsf{T}}	
\newcommand{\vol}{{\textnormal{vol}_{g}\,}} 

\begin{document}
%
%
%
%
\begin{center}
\vspace*{-1cm}

{\Large{\bf  On boundary conditions for linearised Einstein's equations}
\\[1ex]
\small \textit{Applied Mathematics Letters} (2024), 109210. 
\\
DOI: \href{https://doi.org/10.1016/j.aml.2024.109210}{10.1016/j.aml.2024.109210}
} 

\vspace{5mm}

{\bf by}

\vspace{3mm}
\noindent
{  \bf Matteo Capoferri$^{1,2}$, Simone Murro$^{3,4}$ and Gabriel Schmid$^{3}$}\\[2mm]
\noindent   {\it $^1$ Department of Mathematics, Heriot-Watt University, Edinburgh, UK}\\[2mm]
\noindent   {\it $^2$  Maxwell Institute for Mathematical Sciences, Edinburgh, UK}\\[2mm]
\noindent   {\it $^3$ Dipartimento di Matematica,  Universit\`a di Genova \& INdAM, Italy}\\[2mm]
\noindent   {\it $^4$ Istituto Nazionale di Fisica Nucleare, Sezione di Genova, Genova, Italy}\\[2mm]

Emails: \ {\tt m.capoferri@hw.ac.uk, murro@dima.unige.it, gabriel.schmid@dima.unige.it}
\\[10mm]
\end{center}
\begin{abstract}
    We investigate the properties of a fairly large class of boundary conditions for the linearised Einstein equations in the Riemannian setting, ones which generalise the linearised counterpart of boundary conditions proposed by Anderson. Through the prism of the quest to quantise gravitational waves in curved spacetimes, we study their properties from the point of view of ellipticity, gauge invariance, and the existence of a spectral gap.
\end{abstract}
\tableofcontents
\section{Introduction}
\label{Introduction}
Einstein's equations of General Relativity are a cornerstone of modern science. Yet, boundary conditions for these equations are not very well understood, owing to their rather complicated mathematical structure.
In the nonlinear setting, a natural choice of boundary conditions for Einstein's equations in Riemannian signature is a (nonlinear) analogue of the harmonic gauge condition, together with prescribing either the induced boundary metric (\emph{Dirichlet conditions}) or the second fundamental form of the boundary (\emph{von Neumann conditions}). However, it is well known that both of these boundary conditions do not lead to a well-defined elliptic boundary value problem, see e.g.~\cite{Anderson2008,AvramidiEsposito}. Furthermore, they are not gauge invariant (in an appropriate sense to be specified later). Hence, one is forced to look for alternatives. Good candidates --- ones that implement ellipticity --- have been studied by Anderson \cite{Anderson2008,Anderson2010},  Schlenker \cite{Schlenker} and more recently by Liu--Santon--Wiseman \cite{LiuSantonWiseman}. More precisely, Anderson proposed in \cite{Anderson2008} a new set of boundary conditions for the nonlinear Einstein equations in Riemannian signature, which he showed to be elliptic, i.e.~to satisfy the \textit{Shapiro-Lopatinskij conditions} (see e.g.~\cite[Chap.~XX]{Hormander3} or \cite{KrupchykTuomela}), consisting in prescribing: (i) a harmonic gauge condition on the boundary, (ii) the conformal class of the boundary metric and (iii) the mean curvature (the trace of the second fundamental form) of the boundary. Henceforth, we will refer to these as \textit{Anderson boundary conditions}. 

In this paper, we are concerned with the study of boundary conditions for the \emph{linearised} Einstein equations in Riemannian signature, which are a pivotal ingredient in the quantisation of gravitational waves~\cite{GerardMurroWrochna23}. 
Indeed, as argued for instance by Witten in~\cite{Witten,Witten2023}, it turns out that the Anderson boundary conditions are relevant in Euclidean quantum gravity.
Let us emphasise that, despite being a problem formulated within a linear theory, the quantisation of gravitational waves propagating in a generic spacetime satisfying Einstein's field equation is still an open problem. The main reason for this is the lack of a rigorous construction of a ``physical'' state, namely a positive gauge invariant linear functional on the space of gauge invariant observables which satisfies the so-called \emph{Hadamard condition}. Recently, a new strategy based on Wick rotation coupled with elliptic theory on manifolds with boundary and microlocal analysis was explored in~\cite{GerardMurroWrochna23}. The essential idea is
to reduced the construction of a state to the construction of a Calder\'on projector for the elliptic operator 
$ \D_{2}:=-\Delta_{2}+2\,\mathrm{Riem}_{g} $
acting on symmetric $(0,2)$-tensors $u=(u_{\alpha\beta})$,
where $\Delta_{2}=g^{\alpha\beta}\nabla_{\alpha}\nabla_{\beta}$ is the connection Laplacian and $\mathrm{Riem}_{g}$ the \textit{Riemann operator}\footnote{$\mathrm{Riem}_{g}(u)_{\alpha\beta}:=\tensor{R}{^\gamma_\alpha_\beta^\delta}u_{\gamma\delta}$, where we adopt Einstein's summation notation throughout and the convention $(\nabla_{\alpha}\nabla_{\beta}-\nabla_{\beta}\nabla_{\alpha})\omega_{\gamma}=\tensor{R}{_\alpha_\beta_\gamma^\delta}\omega_{\delta}$.} on the Riemannian manifold with boundary $(\M:=[-\T,\T]\times \Sigma,g=\d s\otimes\d s+\gamma_s)$, with $\T>0$ and $\Sigma$ is a smooth connected oriented complete Riemannian $3$-manifold with $\partial\Sigma=\emptyset$. Clearly, the construction of the Calder\'on projector depends on the choice of boundary conditions at $s=\pm\T$. For the construction to be physically meaningful, the boundary conditions, on top of being elliptic, should at the very least satisfy:
(a) a \emph{spectral gap condition}, i.e.~the spectrum of $\D_2$ should not contain $0$;
(b) a \emph{gauge invariance condition}. Let us emphasise that ellipticity not only ensures that the boundary value problem is well behaved, but is also crucial for perturbation theory, see e.g.~\cite[Sec.~2.1.]{Witten}. In~\cite{GerardMurroWrochna23}, the authors considered Dirichlet boundary conditions for $u_{\alpha\beta}$ at $s=\pm\T$: this choice guarantees the presence of a spectral gap, but does not allow one to achieve gauge invariance, hence the quest remains open. 
\medskip 

All in all, the above considerations and challenges are the starting point and motivation for this work.\bigskip

\noindent\textbf{Main results.} Let $(\M,g)$ be a smooth connected oriented Riemannian $4$-manifold with non-empty boundary $\partial\M$, solving Einstein's equations $\mathrm{Ric}(g)-\frac{1}{2}\mathrm{Scal}(g)g+\Lambda g=0$ for some choice of the \textit{cosmological constant} $\Lambda\in\RR$. Furthermore, let  $\V_{k}:=\CC\otimes\T^{\ast}\M^{\otimes_{s}k}$ be the $\CC$-vector bundle of symmetric $(0,k)$-tensors equipped with the inner product
\begin{align*}
     (\,\cdot\,,\,\cdot\,)_{\V_{k}}:=k!\int_{\M}\,(g^{-1})^{\otimes k}(\overline{\,\cdot\,},\,\cdot\,)\,\vol
\end{align*}
on compactly-supported smooth sections $\Gamma_{\c}(\V_{k})$. Denoting by $\L\:\Gamma(\V_{2})\to\Gamma(\V_{2})$ the linearisation of Einstein's equations, we define the operator
\begin{align*}
	\P\:\Gamma(\V_{2})\to\Gamma(\V_{2}),\qquad \P:=2(\L\circ\I)=-\Delta_{2}+2\mathrm{Riem}_{g}-\I\circ\d\circ\delta\,,
\end{align*}
where the involution $\I:=\mathrm{id}-\frac{1}{2}\mathrm{tr}_{g}(\cdot)g$ denotes the \textit{trace-reversal}, $(\delta u)_{\mu}=-2\nabla^{\lambda}u_{\lambda\mu}$ is the \textit{divergence} of $u\in\Gamma(\V_{2})$, and $(\d \omega)_{\alpha\beta}= \frac{1}{2}(\nabla_{\alpha}\omega_{\beta}+\nabla_{\beta}\omega_{\alpha})$  the \textit{symmetrised gradient} of $\omega\in\Gamma(\V_{1})$. By construction, $\P$ is formally self-adjoint w.r.t.~$(\,\cdot\,,\I\,\cdot\,)_{\V_{2}}$. Diffeomorphism invariance of Einstein's equations translates into the fact that $\P$ is invariant under the linear \textit{gauge transformation} $u\mapsto u+\K\omega$, where $\K:=\I\circ\d\:\Gamma(\V_{1})\to\Gamma(\V_{2})$.
Imposing the gauge condition $\delta u=0$, known as \textit{harmonic} or \textit{de Donder gauge}, reduces $\P$ to the elliptic operator $\D_{2}$ from the previous section, and the remaining gauge freedom is parameterised by the quantity $\omega$ satisfying $\D_{1}\omega=0$, with $\D_{1}:=\delta\K=-\Delta_{1}-\Lambda$. Note that $\K$ intertwines $\D_{1}$ and $\D_{2}$, i.e. $\K\D_{1}=\D_{2}\K$ --- see \cite{GerardMurroWrochna23} for further details.\medskip

In this paper: (i)~We examine the linearised version of Anderson's boundary conditions, rectifying some inaccuracies in the existing literature (Section~\ref{Linearised Anderson Boundary Conditions}). (ii)~We establish the most general version of Anderson-type conditions which ensure \textit{gauge invariance} in the sense of Definition~\ref{def:gauge invariance} (Section~\ref{General Conformal Boundary Conditions}). (iii)~In the guise of a no-go theorem, we demonstrate that, except for special geometries, one cannot achieve the spectral gap condition for the latter (Section~\ref{Spectral Theory of Linearised Anderson Boundary Conditions}).\medskip

Our main result is concisely summarised by the following theorem.

\begin{main-theorem}
Let $(\Sigma,\gamma)$ be a smooth connected oriented complete Riemannian $3$-manifold with $\partial\Sigma=\emptyset$. For $T>0$ consider the Riemannian $4$-manifold $\Omega:=[-\T,\T]\times\Sigma$ equipped with the metric $g:=\d s\otimes\d s+\gamma$ and let $\D_{2}$ be the linearised Einstein operator on $\Omega$.
\begin{itemize}
    \item[(i)] Suppose that the Laplace--Beltrami operator on $(\Sigma,\gamma)$ has no spectral gap. Then $\D_{2}$ supplied with the linearised Anderson boundary conditions on $\partial\Omega$ has no spectral gap, namely, $0\in\sigma(\D_{2})$\footnote{Here and further on $\sigma(\A)$ denotes the spectrum of the operator $\A$.}.
    \item[(ii)] Suppose that $\mathrm{Ric}(\gamma)=0$ and there exist non-trivial $\L^{2}$-harmonic $1$-forms on $\Sigma$. If $\D_{2}$ is supplied with any first-order, linear, fully determined, elliptic and gauge invariant boundary conditions which include the conditions\footnote{Here $u_{\Sigma\Sigma}:=u\vert_{\T\Sigma}$ denotes the tangential part of $u\in\Gamma(\V_{2})$.}
\begin{align}
\label{main theorem equation 1}
    \delta u=0\qquad\text{and}\qquad u_{\Sigma\Sigma}=\frac{1}{3}\mathrm{tr}_{\gamma}(u_{\Sigma\Sigma})\gamma\quad\text{on}\quad\partial\Omega=\{\pm\T\}\times\Sigma \,,
\end{align}
then $0\in\sigma(\D_{2})$.
 \end{itemize}
\end{main-theorem}

The two conditions in~\eqref{main theorem equation 1} correspond to the harmonic gauge condition and the linearisation of the requirement that the conformal class of the boundary metric $\gamma$ is fixed, respectively. This provides nine independent conditions on $u\in\Gamma(\V_{2})$; since $\mathrm{rank}_{\CC}(\V_{2})=10$, we are free to add one more (scalar) condition to obtain a fully-determined boundary value problem. If one imposes the vanishing of the linearisation of the mean curvature of the boundary, one recovers the linearised Anderson conditions as a special case. 
\section{Linearised Anderson boundary conditions}
\label{Linearised Anderson Boundary Conditions}

Let $(\M,g)$ be a Riemannian manifolds with non-empty boundary as in Section~\ref{Introduction}. At the non-linear level, the Anderson boundary conditions consist in imposing a non-linear analogue of the harmonic gauge, see e.g.~\cite[Chapter~7]{choquet-bruhat_general_2009}, together with 
\begin{align}\label{eq:Anderson}
\begin{cases}
	[\gamma]&=\text{fixed},\\ 
	\mathrm{tr}_{\gamma}(k)&=\text{fixed},
\end{cases}
\end{align}
where $\gamma:=g\vert_{\partial\M}$ denotes the induced boundary metric, $[\gamma]$ its conformal class and $k\in\Gamma(\T\partial\M^{\otimes_{s}2})$ the corresponding second fundamental form. To linearise these conditions, let us consider a one-parameter family of Riemannian metrics $\hat{g}(\lambda)$ on $\M$ with $\hat{g}(0)=g$, and formally expand $\hat{g}(\lambda)=g+h\lambda+\mathcal{O}(\lambda^{2})$. In the following, all the quantities and operations associated with the ``full'' metric $\hat{g}$ will be denoted with a hat, whereas objects associated with the background metric $g$ without a hat. Let us choose a local coordinate frame of $\partial\M$, denoted by $(\partial_{i})_{i=1,2,3}$, and complete it to a local frame on $\M$ by adding a local vector field $\partial_{0}$ not tangential to $\partial\M$. We will denote the full frame by $(\partial_{\mu})_{\mu=0,1,2,3}$. In this frame, the induced metric $\hat{\gamma}:=\hat{g}\vert_{\partial\M}$ has components $\hat{g}_{ij}$; the remaining components $\hat{g}_{0\nu}$   are normal to $\partial\M$. 

\begin{proposition}\label{Prop}
    Consider the linearisation $\hat{g}(\lambda)=g+h\lambda+\mathcal{O}(\lambda^{2})$ and choose coordinates in which the background metric is block-diagonal: $g_{00}=1$ and $g_{0i}=0$. Then the linearisation of \eqref{eq:Anderson} reads
	\begin{align*}\label{eq:linAnd}
	\begin{cases}
		h_{ij}-\frac{1}{3}\mathrm{tr}_{\gamma}(h^{\T})\,\gamma_{ij}=0\\
		\nabla_{0}\mathrm{tr}_{\gamma}(h^{\T})-2\nabla^{i}h_{i0}+2k^{ij}h_{ij}-h_{00}\,\mathrm{tr}_{\gamma}(k)=0
		\end{cases}\quad\text{on}\quad\partial\M\,,
	\end{align*}
	where $\gamma:=g\vert_{\partial\M}$ and $\mathrm{tr}_{\gamma}(h^{\T})=\gamma^{ij}h_{ij}$ is
	 the trace of the tangential parts of $h$.
\end{proposition}

\begin{proof}
	 It is clear that fixing the conformal class of the boundary metric $\hat{\gamma}$ translates into the requirement that the traceless part of its linearisation $h^{\T}$ vanishes, i.e.~$h_{ij}-\frac{1}{3}(\gamma^{kl}h_{kl})\gamma_{ij}=0$. To linearise the equation $\mathrm{tr}_{\hat{\gamma}}(\hat{k})=\text{fixed}$, let $\hat{n}=\hat{n}^{\mu}\partial_{\mu}$ be the unit normal vector field on $\partial\M$. Since $\hat{g}(\hat{n},\partial_{i})=0$ and $\hat{g}(\hat{n},\hat{n})=1$, by definition, we find that the \textit{covariant} components of $\hat{n}$ are given by $\hat{n}_{i}=0$ and $\hat{n}_{0}=(\hat{g}^{00})^{-\frac{1}{2}}$. It follows that the second fundamental form, which is defined by $\hat{k}(X,Y):=\hat{g}(\hat{n},\hat{\nabla}_{X}Y)$ for all $X,Y\in\Gamma(\T\M)$ tangential to $\partial\M$, in these coordinates is given by $\hat{k}_{ij}=-\hat{\nabla}_{j}\hat{n}_{i}=\hat{\Gamma}_{ij}^{0}(\hat{g}^{00})^{-\frac{1}{2}}$. We recall that inverse metric and Christoffel symbols under perturbation $\hat{g}_{\mu\nu}=g_{\mu\nu}+\lambda h_{\mu\nu}+\mathcal{O}(\lambda^{2})$ expand as $\hat{g}^{\mu\nu}=g^{\mu\nu}-h^{\mu\nu}\lambda+\mathcal{O}(\lambda^{2})$ and $\hat{\Gamma}_{\alpha\beta}^{\gamma}=\Gamma_{\alpha\beta}^{\gamma}+\frac{1}{2}(\nabla_{\alpha}{h^{\gamma}}_{\beta}+\nabla_{\beta}{h^{\gamma}}_{\alpha}-\nabla^{\gamma}h_{\alpha\beta})\lambda+\mathcal{O}(\lambda^{2})$, see e.g.~\cite[Ch.~1, eq.~11.3 and 11.4]{choquet-bruhat_general_2009}, where the indices of $h$ are raised/lowered with the background metric $g$. Using these relations and choosing local coordinates such that the background metric reads $g_{00}=1$ and $g_{0i}=0$, one arrives at
	\begin{align*}
		\mathrm{tr}_{\hat{\gamma}}(\hat{k})=\hat{g}^{ij}\hat{k}_{ij}=\mathrm{tr}_{\gamma}(k)+\bigg\{\nabla^{i}h_{i0}-\frac{1}{2}\nabla_{0}\mathrm{tr}_{\gamma}(h^{\T})-h^{ij}k_{ij}+\frac{1}{2}\mathrm{tr}_{\gamma}(k)h_{00}\bigg\}\lambda+\mathcal{O}(\lambda^{2})\, .
	\end{align*}
	The scalar boundary condition requires $\mathrm{tr}_{\hat{\gamma}}(\hat{k})=\mathrm{tr}_{\gamma}(k)$ and hence, we obtain the claimed result. 
\end{proof} 

\begin{remark}
	Note that the linearised Anderson boundary conditions have also been considered in the recent review \cite[Section 3.3]{Witten}.  Now, the linearised expressions from \cite[Section 3.3]{Witten} appear to be incomplete, in that the term $-h_{00}\mathrm{tr}_{\gamma}(k)$ is missing. Whilst this does not invalidate the discussion in \cite{Witten}, because the extra term vanishes in the Euclidean setting $(\Sigma,\gamma)=(\mathbb{R}^{3},\delta)$, this term is crucial to achieve gauge invariance in curved space, as we will see below. This justifies why a detailed derivation was warranted.
\end{remark}

\begin{example}
\label{example}
As a prototypical example, let us consider a manifold $\Omega=[-\T,\T]\times\Sigma$ with metric $g=\d s\otimes\d s+\gamma_{s}$,  where $\Sigma$ is a (connected, oriented) smooth $3$-manifold,  $s\:\M\to\RR$ is the \textit{Euclidean time}, and $\gamma_{s}$ a one-parameter family of Riemannian metrics on $\Sigma$.  Any tensor $u\in\Gamma(\V_{2})$ can be decomposed as
\begin{align*}u=u_{ss}\,\d s\otimes \d s+2u_{s\Sigma}\otimes\d s+u_{\Sigma\Sigma}\,,\end{align*} where: $u_{ss}:=u(\partial_{s},\partial_{s})\in C^{\infty}(\RR,\Gamma(\V_{\Sigma,0}))$ is a $s$-dependent smooth function on $\Sigma$, $u_{s\Sigma}:=u(\partial_{s},\cdot)\vert_{\T\Sigma}\in C^{\infty}(\RR,\Gamma(\V_{\Sigma,1}))$ a $s$-dependent (0,1)-tensor on $\Sigma$, and $u_{\Sigma\Sigma}\in C^{\infty}(\RR,\Gamma(\V_{\Sigma,2}))$ a $s$-dependent symmetric (0,2)-tensor on $\Sigma$. We used the notation $\V_{\Sigma,k}:=\CC\otimes\T^{\ast}\Sigma^{\otimes_{s}k}$.
 Through the prism of this decomposition of $u$, let us write down the linearised Anderson boundary conditions at $\partial\Omega=\{\pm\T\}\times\Sigma$ for $\D_{2}$. Recall that $\D_{2}$ originates from $\P\propto\L\circ\I$ upon harmonic gauge fixing, so that one needs to replace $h$ in Proposition~\ref{Prop} with $h=:\I u$. 
 A straightforward computation shows that the linearised Anderson boundary conditions on $\partial\Omega$ take the form on $\partial\Omega$
 \begin{equation}\label{eq:AndersonBC}
\begin{aligned}
-\frac{1}{2}\delta u=&\begin{cases}\begin{aligned}
		&(a)\quad\partial_{s}u_{ss}-\delta_{\Sigma}u_{s\Sigma}-\mathrm{tr}_{\gamma}(k)u_{ss}+(\gamma^{-1})^{\otimes 2}(k,u_{\Sigma\Sigma})=0\\
		&(b)\quad\partial_{s}u_{s\Sigma}-\frac{1}{2}\delta_{\Sigma}u_{\Sigma\Sigma}-\mathrm{tr}_{\gamma}(k)u_{s\Sigma}
	=0
 \end{aligned}\end{cases}\\
&\begin{cases}\begin{aligned}
	&(c)\quad u_{\Sigma\Sigma}-\frac{1}{3}\mathrm{tr}_{\gamma}(u_{\Sigma\Sigma})\gamma=0
	\\
	&(d)\quad 3\mathrm{tr}_{\gamma}(k) \bigg (u_{ss}+\frac{1}{3}\mathrm{tr}_{\gamma}(u_{\Sigma\Sigma})\bigg)+\partial_{s}\mathrm{tr}_{\gamma}(u_{\Sigma\Sigma})-\partial_{s}u_{ss}-\frac{4}{3}\mathrm{tr}_{\gamma}(u_{\Sigma\Sigma})\mathrm{tr}_{\gamma}(k)=0
\end{aligned}\end{cases}
\end{aligned}
 \end{equation}
where $\delta_{\Sigma}$ denotes the divergence\footnote{Namely, $\delta_{\Sigma}\omega=-\nabla_{\Sigma}^{i}\omega_{i}$ and $(\delta_{\Sigma}u)_{j}=-2\nabla_{\Sigma}^{i}u_{ij}$ for $\omega\in\Gamma(\V_{\Sigma,1})$ and $u\in\Gamma(\V_{\Sigma,2})$, respectively, where $\nabla_{\Sigma}$ is the Levi-Civita connection of $\Sigma$.} on $(\Sigma,\gamma)$
 and $k:=-\frac{1}{2}\partial_{s}\gamma\in C^{\infty}(\RR,\Gamma(\V_{\Sigma,2}))$ is the second fundamental form of $\Sigma$.
 \end{example}
\section{General conformal boundary conditions}
\label{General Conformal Boundary Conditions}

The aim of this section is to establish the most general  \emph{gauge invariant} boundary condition for $\D_2$ which satisfies the requirements (a)-(c) in~\eqref{eq:AndersonBC} on a Riemannian manifold of the form $\Omega=[-\T,\T]\times\Sigma$, $\T>0$. 

\begin{definition}\label{def:gauge invariance}
A set of boundary conditions $\mathfrak{B}_2$ for $\D_2$ is said to be \emph{gauge invariant} if, chosen boundary conditions $\mathfrak{B}_1$ for $\D_1$, the following holds: for any $\omega\in\Gamma(\V_{1})$ satisfying $\mathfrak{B}_1$ and $\D_{1}\omega=0$ near $\partial\Omega$,
 the tensor $u:=\K\omega$ satisfies $\mathfrak{B}_2$. 
\end{definition}

The condition $\D_{1}\omega=0$ near $\partial\Omega$ ensures that the harmonic gauge condition is preserved under gauge transformations. In the forthcoming discussion, we will restrict ourselves to the case where $\mathfrak{B}_1$ are Dirichlet boundary conditions, namely $\omega|_{\partial\Omega}=0$. This is a very natural choice for geometric reasons: it is the linearised version of invariance under diffeomorphisms $\phi\in\mathrm{Diff}(\Omega)$ that fix the boundary,  $\phi\vert_{\partial\Omega}=\mathrm{id}_{\partial\Omega}$.

\begin{proposition}\label{prop2}
	Let $\Omega:=[-\T,\T]\times\Sigma$, $\T>0$, be a Riemannian manifold with metric $g=\d s\otimes\d s+\gamma_{s}$, and consider the boundary conditions
	\begin{align}\label{anderson most general equation 1}
\begin{cases}
		(a)\quad\partial_{s}u_{ss}-\delta_{\Sigma}u_{s\Sigma}-\mathrm{tr}_{\gamma}(k)u_{ss}+(\gamma^{-1})^{\otimes 2}(k,u_{\Sigma\Sigma})=0\\
		(b)\quad\partial_{s}u_{s\Sigma}-\frac{1}{2}\delta_{\Sigma}u_{\Sigma\Sigma}-\mathrm{tr}_{\gamma}(k)u_{s\Sigma}
	=0\\
		(c)\quad u_{\Sigma\Sigma}-\frac{1}{3}\mathrm{tr}_{\gamma}(u_{\Sigma\Sigma})\gamma=0
\end{cases}\quad\text{on}\quad\partial\Omega
\end{align}
for the operator $\D_{2}$. The most general first order linear scalar boundary condition that  can be appended to (a)-(c) in such a way that the resulting set of boundary conditions is gauge invariant in the sense of Definition~\ref{def:gauge invariance} reads on $\partial\Omega$
\begin{equation}
\label{anderson most general equation 2}
\begin{aligned}
	(d)\quad \C_{1}\bigg(u_{ss}+\frac{1}{3}\mathrm{tr}_{\gamma}(u_{\Sigma\Sigma})\bigg)+\C_{2}\bigg(\partial_{s}\mathrm{tr}_{\gamma}(u_{\Sigma\Sigma})-\partial_{s}u_{ss}-\frac{4}{3}\mathrm{tr}_{\gamma}(u_{\Sigma\Sigma})\mathrm{tr}_{\gamma}(k)\bigg)& \\ +(\gamma^{-1})\bigg(\V,\d_{\Sigma}\Big(u_{ss}+\frac{1}{3}\mathrm{tr}_{\gamma}(u_{\Sigma\Sigma})\Big)\bigg)+(\gamma^{-1})^{\otimes 2}\bigg(\S,\partial_{s}u_{\Sigma\Sigma}-2\d_{\Sigma}u_{s\Sigma}&\\
	-\gamma\Big(\partial_{s}u_{ss}+\frac{2}{3}\mathrm{tr}_{\gamma}(u_{\Sigma\Sigma})\mathrm{tr}_{\gamma}(k)\Big)\bigg)&=0
 \end{aligned}
\end{equation}
for some coefficients $\C_{i}\in C^{\infty}(\partial\Omega)$, $\V\in\Gamma(\T^{\ast}\partial\Omega)$ and $\S\in\Gamma(\T^{\ast}\partial\Omega^{\otimes_{s}2})$, with $\S$ not proportional to $\gamma\vert_{\partial\Omega}$.
\end{proposition}

\begin{proof}
	Let $\omega\in\Gamma(\V_{1})$ be such that $\omega\vert_{\partial\Omega}=0$ and $\D_{1}\omega=0$ in a neighbourhood of $\partial\Omega$. We decompose $\omega=\omega_{s}\d s+\omega_{\Sigma}$, where $\omega_{s}:=\omega(\partial_{s})\in C^{\infty}(\RR,\Gamma(\V_{\Sigma,0}))$ is a $s$-dependent function on $\Sigma$ and $\omega_{\Sigma}:=\omega-\omega_{s}\d s\in C^{\infty}(\RR,\Gamma(\V_{\Sigma,1}))$ is a $s$-dependent (0,1)-tensor on $\Sigma$. A straightforward computation shows that $u:=\K\omega$ and $\omega$ on $\partial\Omega$ are related as
	\begin{align}\label{eq:Relations}
	\begin{cases}
		u_{ss}=\frac{1}{2}\partial_{s}\omega_{s}\\
		u_{s\Sigma}=\frac{1}{2}\partial_{s}\omega_{\Sigma}\\
		u_{\Sigma\Sigma}=-\frac{1}{2}\gamma\partial_{s}\omega_{s}\\
		\mathrm{tr}_{\gamma}(u_{\Sigma\Sigma})=-\frac{3}{2}\partial_{s}\omega_{s}
	\end{cases}\text{and}\qquad
	\begin{cases}
		\partial_{s}u_{ss}=\frac{1}{2}(\partial_{s}^{2}\omega_{s}+\partial_{s}\delta_{\Sigma}\omega_{\Sigma}+\mathrm{tr}_{\gamma}(k)\partial_{s}\omega_{s})\\
		\partial_{s}u_{s\Sigma}=\frac{1}{2}(\partial_{s}^{2}\omega_{\Sigma}+\partial_{s}\d_{\Sigma}\omega_{s}+2k\cdot\partial_{s}\omega_{\Sigma})\\
		\partial_{s}u_{\Sigma\Sigma}=\partial_{s}\d_{\Sigma}\omega_{\Sigma}-\frac{1}{2}\gamma(\partial_{s}^{2}\omega_{s}-\partial_{s}\delta_{\Sigma}\omega_{\Sigma}-\mathrm{tr}_{\gamma}(k)\partial_{s}\omega_{s})\\
	\partial_{s}\mathrm{tr}_{\gamma}(u_{\Sigma\Sigma})=-\frac{3}{2}\partial_{s}^{2}\omega_{s}+\frac{1}{2}(\mathrm{tr}_{\gamma}(k)\partial_{s}\omega_{s}+\partial_{s}\delta_{\Sigma}\omega_{\Sigma})
	\end{cases}
\end{align}
where for any $\eta\in\Gamma(\V_{\Sigma,1})$ the term $k\cdot\eta$ denotes the $(0,1)$-tensor $(k\cdot\eta)_{i}:={k_{i}}^{k}\eta_{k}$ and $\d_{\Sigma}$ denotes the symmetrised gradient on $(\Sigma,\gamma)$. 
Note that (c) is clearly a gauge invariant condition. Furthermore, the condition $\D_{1}\omega=0$ on $\partial\Omega$ translates into the equations
\begin{align}\label{eq:2}
	\D_{1}\omega=\delta u=\begin{cases}
		-\partial_{s}^{2}\omega_{s}+\mathrm{tr}_{\gamma}(k)\partial_{s}\omega_{s}=0\\
		-\partial_{s}^{2}\omega_{\Sigma}-2k\cdot\partial_{s}\omega_{\Sigma}+\mathrm{tr}_{\gamma}(k)\partial_{s}\omega_{\Sigma}=0
	\end{cases}\quad\text{on}\quad\partial\Omega
\end{align}
corresponding to the boundary conditions (a) and (b) combined. We would like to write down the most general 1st order, linear, scalar and gauge invariant boundary condition one can add to (a)-(c). This scalar boundary conditions can include all the 0th and 1st order terms one can construct out of $u_{ss},u_{s\Sigma}$ and $u_{\Sigma\Sigma}$, i.e.~$u_{ss}$, $\partial_{s}u_{ss}$, $\mathrm{tr}_{\gamma}(u_{\Sigma\Sigma})$, $\partial_{s}\mathrm{tr}_{\gamma}(u_{\Sigma\Sigma})$, $\delta_{\Sigma}u_{s\Sigma},u_{s\Sigma}$, $\partial_{s}u_{s\Sigma}$, $\d_{\Sigma}u_{ss}$, $\d_{\Sigma}\mathrm{tr}_{\gamma}(u_{\Sigma\Sigma})$, $\delta_{\Sigma}u_{\Sigma\Sigma}$, $u_{\Sigma\Sigma}$, $\partial_{s}u_{\Sigma\Sigma}$, $\d_{\Sigma}u_{s\Sigma}$ and $\d_{\Sigma}u_{\Sigma\Sigma}$. Since we have already fixed five conditions on the boundary --- the boundary conditions (a), (b) and (c) together with $\delta_{\Sigma}$(c) and $\d_{\Sigma}$(c) --- we can drop five of them, since they are not independent of the others. Therefore, we make the general Ansatz 
\begin{align*}
	&\C_{1}u_{ss}+\C_{2}\partial_{s}u_{ss}+\C_{3}\mathrm{tr}_{\gamma}(u_{\Sigma\Sigma})+\C_{4}\partial_{s}\mathrm{tr}_{\gamma}(u_{\Sigma\Sigma})+\gamma^{-1}(\V_{1},u_{s\Sigma})+\gamma^{-1}(\V_{2},\d_{\Sigma}u_{ss})\\
	&\quad {+\gamma^{-1}(\V_{3},\d_{\Sigma}\mathrm{tr}_{\gamma}(u_{\Sigma\Sigma}))+(\gamma^{-1})^{\otimes 2}(\S_{1},\d_{\Sigma}u_{s\Sigma})+(\gamma^{-1})^{\otimes 2}(\S_{2},\partial_{s}u_{\Sigma\Sigma})=0\quad\text{on}\quad\partial\Omega}
\end{align*}
for coefficients $\C_{i}\in C^{\infty}(\partial\Omega)$, $\V_{i}\in\Gamma(\T^{\ast}\partial\Omega)$ and $\S_{i}\in\Gamma(\T^{\ast}\partial\Omega^{\otimes_{s}2})$ with $\S_{i}$ not proportional to  $\gamma\vert_{\partial\Omega}$. Taking $u=\K\omega$ and using \eqref{eq:Relations} together with \eqref{eq:2} to replace the $\partial_{s}^{2}$-terms, we see that gauge invariance requires one to have
\begin{align*}
	&\frac{1}{2}(\C_{1}-3\C_{3}+2\mathrm{tr}_{\gamma}(k)(\C_{2}-\C_{4}))\partial_{s}\omega_{s}+\frac{1}{2}(\C_{2}+\C_{4}+\mathrm{tr}_{\gamma}(\S_{2}))\delta_{\Sigma}\partial_{s}\omega_{\Sigma}+\frac{1}{2}\gamma^{-1}(\V_{1},\partial_{s}\omega_{\Sigma})+\\&\pushright{+\frac{1}{2}\gamma^{-1}(\V_{2}-3\V_{3},\d_{\Sigma}\partial_{s}\omega_{s})+\frac{1}{2}(\gamma^{-1})^{\otimes 2}(\S_{1}+2\S_{2},\d_{\Sigma}\partial_{s}\omega_{\Sigma})\stackrel{!}{=}0\quad\text{on}\quad\partial\Omega}\, .
\end{align*}
Since this has to be true for every $\omega$, the coefficients in front of each term have to vanish identically.  It then follows that the most general scalar gauge invariant condition takes the form 
\begin{align*}
	\C_{1}\bigg(u_{ss}+\frac{1}{3}\mathrm{tr}_{\gamma}(u_{\Sigma\Sigma})\bigg)&+\C_{2}\bigg(\partial_{s}\mathrm{tr}_{\gamma}(u_{\Sigma\Sigma})-\partial_{s}u_{ss}-\frac{4}{3}\mathrm{tr}_{\gamma}(u_{\Sigma\Sigma})\mathrm{tr}_{\gamma}(k)\bigg)\\
	& + (\gamma^{-1})\Big(\V,\d_{\Sigma}\bigg(u_{ss}+\frac{1}{3}\mathrm{tr}_{\gamma}(u_{\Sigma\Sigma})\Big)\bigg)+(\gamma^{-1})^{\otimes 2}\bigg(\S,\partial_{s}u_{\Sigma\Sigma}-2\d_{\Sigma}u_{s\Sigma}\\
	&-\gamma\Big(\partial_{s}u_{ss}+\frac{2}{3}\mathrm{tr}_{\gamma}(u_{\Sigma\Sigma})\mathrm{tr}_{\gamma}(k)\Big)\bigg)=0\quad\text{on}\quad\partial\Omega
\end{align*}
for coefficients $\C_{i}\in C^{\infty}(\partial\Omega)$, $\V\in\Gamma(\T^{\ast}\partial\Omega)$ and $\S\in\Gamma(\T^{\ast}\partial\Omega^{\otimes_{s}2})$ with $\S$ not proportional to $\gamma\vert_{\partial\Omega}$.
\end{proof}

\begin{example}\label{Example:Anderson}
	The Anderson conditions \eqref{eq:AndersonBC} are a special case of \eqref{anderson most general equation 1}, \eqref{anderson most general equation 2} with $\C_{1}=3\mathrm{tr}_{\gamma}(k)$, $\C_{2}=1$ and $\V=\S=0$. 
\end{example}

\begin{remark}
    In full generality, i.e.~for arbitrary coefficients, the boundary conditions in Proposition~\ref{prop2} are not elliptic. Since ellipticity is a local condition, it is enough to consider the special case $(\Sigma,\gamma)=(\RR^{3},\delta)$, in which case $\D_{2}=-\partial_{s}^{2}-\sum_{i}\partial_{i}^{2}$. To check the \textit{Shapiro-Lopatinskij} (SL) conditions, let $u$ be a smooth solution of $\D_{2}u=0$ which is bounded in $\mathbb{R}^{+}$, i.e.~$u_{\mu\nu}(s,\vec{x})=c_{\mu\nu}e^{\i\xi\cdot x}e^{-\vert\xi\vert s}$ for some $\xi\in\RR^{3}$ and $c=(c_{\mu\nu})_{0\leq\mu,\nu\leq 3}\in\mathbb{C}^{4\times 4}$. Then the SL conditions are satisfied if the trivial solution $c_{\mu\nu}=0$ is the only such solution for arbitrary $\xi\neq 0$ satisfying the boundary conditions. Plugging this solution into the boundary conditions (a)-(c) in the Euclidean case, we obtain
\begin{align}\label{eq:Ell}
	\begin{cases}
		-\vert\xi\vert c_{ss}+\i\xi^{i} c_{si}&=0\\
		-\vert\xi\vert c_{sj}+\i\xi^{i} c_{ij}&=0\\
		c_{\Sigma\Sigma}-\frac{1}{3}\mathrm{tr}_{\delta}(c_{\Sigma\Sigma})\delta&=0\\
	\end{cases}
	\quad\Leftrightarrow\quad
	\begin{cases}
		c_{ss}&=-\frac{1}{3}\mathrm{tr}_{\delta}(c_{\Sigma\Sigma})\\
		c_{s\Sigma}&=\frac{\i}{3}\vert\xi\vert^{-1}\mathrm{tr}_{\delta}(c_{\Sigma\Sigma})\xi\\
		c_{\Sigma\Sigma}&=\frac{1}{3}\mathrm{tr}_{\delta}(c_{\Sigma\Sigma})\delta
	\end{cases}\, .
\end{align}
	Hence, $c_{\mu\nu}=0$ if and only if the spatial trace $\mathrm{tr}_{\delta}(c_{\Sigma\Sigma})=\delta^{ij}c_{ij}$ is zero. The general scalar boundary condition \eqref{anderson most general equation 2} gives the additional condition
 \begin{align}\label{eq:Alg}
     2\C_{2}\vert_{\gamma=\delta}\,\vert\xi\vert\,\mathrm{tr}_{\delta}(c_{\Sigma\Sigma})+\S^{ij}\vert_{\gamma=\delta}\left(\vert\xi\vert\delta_{ij}-\frac{\xi_{i}\xi_{j}}{\vert\xi\vert}\right)\mathrm{tr}_{\delta}(c_{\Sigma\Sigma})=0\, ,
 \end{align}
 where we already used \eqref{eq:Ell} to recast everything in terms of $\mathrm{tr}_{\gamma}(c_{\Sigma\Sigma})$. Note that the coefficients $\C_{2}$ and $\S$ are in principle allowed to depend on the metric $\gamma$ and its derivatives and hence can be zero in the Euclidean case. Now, for ellipticity, $\C_{2}$ and $\S$ have to be chosen such that \eqref{eq:Alg} has no non-trivial solution $\xi\in\RR^{3}$. This is for example the case for the linearised Anderson condition ($\S=0$ and $\C_{2}=1$), cf.~Example~\ref{Example:Anderson}, or more generally when $\S\propto k$ and $\C_{2}\neq 0$.
\end{remark}
\section{Spectral theory of general conformal boundary conditions}
\label{Spectral Theory of Linearised Anderson Boundary Conditions}

Let us consider the same model case as in Example~\ref{example}. We denote by $\D_{2,\Omega}$ the operator $\D_{2}$ in the Hilbert space $\L^{2}(\Omega)$ defined to be the completion of $\Gamma_{\c}(\V_{2})$ w.r.t.~the inner product 
\begin{align*}
(u,v)_{\V_{2}}=2\int_{-\T}^{\T}\d s\int_{\Sigma} \bigg(\overline{u}_{ss}v_{ss}+2\gamma^{-1}(\overline{u}_{s\Sigma},v_{s\Sigma})+(\gamma^{-1})^{\otimes 2}(\overline{u}_{\Sigma\Sigma},v_{\Sigma\Sigma})\bigg)\,\mathrm{vol}_{\gamma_{s}}
\end{align*}
and with domain $\mathcal{D}(\D_{2,\Omega})=\{u\in\H^{2}(\Omega)\mid u\text{ satisfies boundary conditions from Proposition~\ref{prop2}}\}$. Note that $\D_{2,\Omega}$ is in general not self-adjoint. For instance, the linearised Anderson conditions are symmetric w.r.t.~$(\,\cdot\,,\I\,\cdot\,)_{\V_{2}}$ rather than $(\,\cdot\,,\,\cdot\,)_{\V_{2}}$, see e.g.~\cite{Witten}.

\

We are now in a position to prove our Main Theorem.

\begin{proof}[Proof of the Main Theorem.]
    First of all, note that $k=0$ since $\gamma$ does not depend on Euclidean time. The Gau\ss-Codazzi equations then imply that $\tensor{R}{^\alpha_\beta_\gamma^\delta}=0$ if at least one of the indices is zero, in which case $\D_{2}u$ decomposes into the three pieces
    \vspace{-2mm}
    \begin{align*}
        (\D_{2}u)_{ss}=-\partial_{s}^{2}u_{ss}-\Delta_{\Sigma,0}u_{ss},\quad  (\D_{2}u)_{s\Sigma}=-\partial_{s}^{2}u_{s\Sigma}-\Delta_{\Sigma,1}u_{s\Sigma}\\
          (\D_{2}u)_{\Sigma\Sigma}=-\partial_{s}^{2}u_{\Sigma\Sigma}-\Delta_{\Sigma,2}u_{\Sigma\Sigma}+2\mathrm{Riem}_{\gamma}(u_{\Sigma\Sigma})\,,
         \vspace{-2mm}
    \end{align*}
     where $\Delta_{\Sigma,k}=\gamma^{ij}(\nabla_{\Sigma})_{i}(\nabla_{\Sigma})_{j}$ is the connection Laplacian of $(\Sigma,\gamma)$ acting in $\Gamma(\V_{\Sigma,k})$. 
    
    Let us first assume that $\Sigma$ satisfies condition (i). Then, Anderson-type boundary conditions $\C_{1}\propto\mathrm{tr}_{\gamma}(k)$, $\C_{2}=\mathrm{const}\neq 0$ and $\V=\S=0$ are trivially fulfilled for any $u\in\Gamma(\V_{2})$ of the form $u_{s\Sigma}=0, u_{\Sigma\Sigma}=0$ and $u_{ss}$ such that $\partial_{s}u_{ss}=0$. Since $\Sigma$ is complete, $\Delta_{\Sigma,0}\:C^{\infty}_{\c}(\Sigma)\to\L^{2}(\Sigma)$ is essentially self-adjoint, and we denote its minimal self-adjoint extension by the same symbol. By assumption, we know that $0\in\sigma(\Delta_{\Sigma,0})$ and since $\Delta_{\Sigma,0}$ is self-adjoint, there exists a Weyl sequence $(\chi_{n})_{n\in\mathbb{N}}\in\mathcal{D}(\Delta_{\Sigma,0})^{\NN}$ for $\lambda=0$, i.e.~$\Vert\chi_{n}\Vert_{\L^{2}}=1$ and $\Vert\Delta_{\Sigma,0}\chi_{n}\Vert_{\L^{2}}\to 0$ as $n\to+\infty$. Note that w.l.o.g.~we can assume that $\chi_{n}\in C^{\infty}_{\c}(\Sigma)$, since $\mathcal{D}(\Delta_{\Sigma,0})$ is the closure of $C^{\infty}_{\c}(\Sigma)$ w.r.t.~the graph norm and one can always choose a Weyl sequence contained in the dense subset $C^{\infty}_{\c}(\Sigma)$. Define $2$-tensors $u_{n}$ by $(u_{n})_{ss}=(4\T)^{-\frac{1}{2}}\chi_{n}$ and $(u_{n})_{s\Sigma}=0$, $(u_{n})_{\Sigma\Sigma}=0$. By construction, $u_{n}\in\mathcal{D}(\D_{2,\Omega})$, $\Vert u_{n}\Vert_{\V_{2}}=1$ and $\Vert\D_{2,\Omega}u_{n}\Vert_{\V_{2}}=\Vert\Delta_{\Sigma,0}\chi_{n}\Vert_{\L^{2}}\xrightarrow{n\to\infty}0$. We conclude that $(u_{n})_{n}$ is a Weyl sequence of $\D_{2,\Omega}$ for $\lambda=0$, which concludes the proof. 
    
    If $\Sigma$ satisfies condition (ii), then there exists a non-zero $\L^{2}$-harmonic 1-form $\omega$ on $\Sigma$. We define $u\in\Gamma(\V_{2})$ by setting $u_{ss}=u_{\Sigma\Sigma}=0$ and $u_{s\Sigma}:=\omega$. By construction $\partial_{s}u_{s\Sigma}=0$. It is well-known that any harmonic form on a complete manifold is closed and coclosed, see e.g.~\cite{Gaffney1954}. In particular, $\delta_{\Sigma}\omega=0$. Furthermore, a result of Yau~\cite[Thm.~6]{Yau1976} implies that any harmonic $1$-form on a complete manifold with non-negative Ricci curvature is parallel; hence, also the symmetrised gradient vanishes, i.e.~$\d_{\Sigma}\omega=0$. Therefore, $u$ fulfills the conditions in Proposition~\ref{prop2} on all of $\Omega$. Now, in the case $\mathrm{Ric}(\gamma)=0$, the de Rham-Hodge Laplacian on $1$-forms agrees with the Laplace-Beltrami operator (by the Weitzenböck identity) and hence $\Delta_{\Sigma,1}\omega=0$. We conclude that $u$ is a non-trivial solution of $\D_{2,\Omega}u=0$.

    The claim of the theorem now follows by combining the above results with Proposition~\ref{prop2}.
\end{proof}


\begin{examples}
    Let us discuss some examples of $3$-manifolds satisfying the assumptions in our Main Theorem. The assumption from item (i) is clearly fulfilled for any compact $\Sigma$, since $\mathrm{ker}(\Delta_{\Sigma,0})=\{\text{constant functions}\}$ in this case. In the non-compact case, a simple example is provided by $\Sigma=\mathbb{R}^{3}$. 
    More generally, if $(\Sigma,\gamma)$ is non-compact, complete, and has Ricci curvature bounded from below, then $0\in\sigma(\Delta_{\Sigma,0})$ if and only if it is \textit{not open at infinity}, i.e.~if there does not exist any constant $C>0$ s.t.~$\mathrm{area}(\partial\mathcal{D})\geq C\cdot\mathrm{vol}(\mathcal{D})$ for any domain $\mathcal{D}$ with smooth compact closure \cite{Buser}. Let us also remark that $\mathrm{ker}(\Delta_{\Sigma,0})=\{0\}$ does in general not imply $0\notin\sigma(\Delta_{\Sigma,0})$ in the non-compact case, as the example of $\mathbb{R}$ shows.

    The assumption from item (ii), i.e.~the existence of non-trivial $\L^{2}$-harmonic $1$-forms, is fulfilled for compact $\Sigma$ if and only if $\Sigma$ is not simply-connected as a consequence of de Rham's Theorem. In the non-compact case, a non-existence theorem was proved in \cite{Dodziuk1981} for complete manifolds with non-negative Ricci tensor and either infinite volume or positive Ricci tensor at at least one point. On the other hand, manifolds with non-trivial harmonic $\L^{2}$-forms can be easily constructed by considering manifolds on which there exists a bounded harmonic functions with finite Dirichlet energy, see e.g.~\cite{Cao1997}.
\end{examples}

%
%
%
%
%
\paragraph{Acknowledgements and funding}
We are most grateful to Lyonell Boulton, Andrea Carbonaro and Christian G\'erard for inspiring discussions related to the content of this paper and to Edward Witten for insightful comments on the manuscript. MC~was supported by EPSRC Fellowship EP/X01021X/1. SM is partially supported by INFN. GS gratefully acknowledges hospitality of the Heriot-Watt University and Maxwell Institute for Mathematical Sciences in Edinburgh. This work is part of the activities of the INdAM-GNFM and was supported by the MIUR Excellence Department Project 2023-2027 awarded to the Department of Mathematics of the University of Genoa.
\end{document}